\documentclass[12pt]{article}%
\usepackage{amsmath}
\usepackage{amsfonts}
\usepackage{amssymb}
\usepackage{graphicx}%
\newtheorem{theorem}{Theorem}

\newenvironment{proof}[1][Proof]{\noindent\textbf{#1.} }{\ \rule{0.5em}{0.5em}}
\begin{document}

\title{Strong truncations and Maximal Ideal Principles}
\author{Karim Boulabiar\\{\small Laboratoire de Recherche LATAO}\\{\small D\'{e}partement de Mathematiques, Facult\'{e} des Sciences de Tunis}\\{\small Universit\'{e} de Tunis El Manar, 2092, El Manar, Tunisia}\\{\small and}\\{\small Mathematical Institute, Leiden University}\\{\small P.O. Box 9512, 2300 RA Leiden, The Netherlands}\\{\small Email: karim.boulabiar@fst.utm.tn }
\and Mohamed Habibi\\{\small Laboratoire de Recherche LATAO}\\{\small D\'{e}partement de Mathematiques, Facult\'{e} des Sciences de Tunis}\\{\small Universit\'{e} de Tunis El Manar, 2092, El Manar, Tunisia}\\{\small Email: mohamed.habibi@ipest.ucar.tn}}
\date{}
\maketitle

\begin{abstract}
We compare two existence principles for maximal ideals, a classical one for
vector lattices with a strong unit and a second, newly introduced one for
vector lattices with a strong truncation. Although the latter strictly
generalizes the former, we show that the two statements are equivalent over ZF
set theory.

\end{abstract}


\medskip

\section{Introduction}

Following Ball \cite{B14}, we call a \emph{truncation} on a vector lattice $X$
any unary operation $\ast$ on the positive cone $X_{+}$ satisfying%
\[
x^{\ast}\wedge y=x\wedge y^{\ast}\quad\text{for all }x,y\in X_{+}.
\]
In \cite{BH20}, such a truncation is said to be \emph{strong} if, for every
$x\in X_{+}$, there exists $\lambda\in\left(  0,\infty\right)  $ such that
$\left(  \lambda x\right)  ^{\ast}=\lambda x$. A basic example arises when $u$
is a positive strong unit of $X$. In this case,%
\[
x^{\ast}=u\wedge x\quad\text{for all }x\in X_{+}%
\]
defines a strong truncation on $X$. Hence, every vector lattice with strong
unit admit a strong truncation. The converse fails in general; for example,
$C\left(  Y\right)  $, where $Y$ is a non-compact locally compact Hausdorff
space, admits plenty of strong truncations but has no strong units.

Let $X$ be a vector lattice with a strong unit $u$. Since no proper ideal of
$X$ contains $u$, the union of any chain of proper ideals is again a proper
ideal. Hence, by Zorn's Lemma, $X$ admits a maximal ideal. The same argument
yields the classical Krull's Theorem, asserting that every commutative ring
with identity has maximal ideals. Despite this formal similarity, the
underlying logical strength of these results differs substantially. Indeed,
Krull's Theorem is equivalent to Zorn's Lemma \cite{B94,H79}, whereas its
vector lattice analogue is equivalent to the strictly weaker Ultrafilter
Theorem \cite{BvR89}.

This naturally raises the question whether a maximal ideal principle persists
in the broader setting of vector lattices equipped with a strong truncation.
As we shall see below, the answer is affirmative. Although the class of vector
lattices with strong truncation is substantially larger than that of vector
lattices with a strong unit, the following two statements are equivalent in ZF
set theory.

\begin{description}
\item[$\mathrm{(T)}$] Every vector lattice with a strong truncation contains a
maximal ideal.

\item[$\mathrm{(U)}$] Every vector lattice with a strong unit contains a
maximal ideal.
\end{description}

The proof is not entirely straightforward. In the absence of a strong unit,
one lacks a canonical reference point on which the usual maximality arguments
rely. In particular, Zorn's Lemma plays no role here, neither in a technical
nor in a conceptual sense.

\section{The proof of the equivalence}

In what follows, $X$ denotes a real vector lattice. In order to make the note
reasonably self-contained, we first recall from \cite{LZ71} the terminology,
notation, and results that will be used in the proof of the main theorem.

A vector subspace $J$ of $X$ is called an \emph{ideal} if, for every $x,y\in
X$, it follows from $\left\vert x\right\vert \leq\left\vert y\right\vert $ and
$y\in J$ that $x\in J$. It is clear that any ideal of $X$ is itself a vector
lattice with respect to the operations and order inherited from $X$. The
\emph{disjoint complement} (or the \emph{polar}) of an element $y\in X_{+}$ is
defined by%
\[
y^{\bot}=\left\{  x\in X:\left\vert x\right\vert \wedge y=0\right\}  .
\]
It is easily checked that $y^{\bot}$ is an ideal of $X$. On the other hand,
the smallest ideal of $X$ containing $y$ is given by%
\[
J_{y}=\left\{  x\in X:\left\vert x\right\vert \leq\lambda y\text{ for some
}\lambda\in\left(  0,\infty\right)  \right\}  .
\]
The element $y$ is said to be a \emph{strong unit} if $X=J_{y}$.

Let $J$ be an ideal of $X$, and let $J\left(  x\right)  $ denote the
equivalence class of any $x\in J$ in the quotient vector space $X/J$.
Accordingly, we use the symbol $J$ to refer, depending on the context, either
to the ideal $J$ itself or to the canonical projection from $X$ onto $X/J$. A
natural order can be defined on $X/J$ by putting%
\[
J\left(  x\right)  \leq J\left(  y\right)  \text{\quad if and only if\quad
}x\leq a+y\text{ for some }a\in J.
\]
Under this order, $X/J$ becomes a vector lattice, and $J$ is a lattice
homomorphism., i.e., a linear map preserving absolute value. It is quite
straightforward to verify that if $\mathcal{J}$ is a (maximal) ideal of $X/J$,
then its preimage%
\[
J^{-1}\left(  \mathcal{J}\right)  =\left\{  x\in X:J\left(  x\right)
\in\mathcal{J}\right\}  .
\]
is a (maximal) ideal of $X\ $containing $J$.

We now prove our main theorem.

\begin{theorem}
The axioms $\mathrm{(T)}$ and $\mathrm{(U)}$ are equivalent.
\end{theorem}

\begin{proof}
The implication $\mathrm{(T)\Rightarrow(U)}$ follows immediately from the
observation in the introduction that any positive strong unit induces a strong
truncation. Let us prove the converse implication $\mathrm{(U)\Rightarrow(T)}$.

Assume $\mathrm{(U)}$, and let $X$ be a vector lattice with a strong
truncation $\ast$. By definition, there exists $u\in X_{+}$ such that
$u^{\ast}<u$. Let $J$ denote the polar of $u-u^{\ast}$, that is,%
\[
J=\left(  u-u^{\ast}\right)  ^{\bot}=\left\{  x\in X:\left\vert x\right\vert
\wedge\left(  u-u^{\ast}\right)  =0\right\}  .
\]
It is clear that $u\notin J$. Let $x\in X_{+}$ and $\lambda\in\left(
0,\infty\right)  $ such that $\left(  \lambda\left\vert x\right\vert \right)
^{\ast}=\lambda\left\vert x\right\vert $ (such a $\lambda$ exists because
$\ast$ is a strong truncation). Set $y=\lambda\left\vert x\right\vert $ and
observe that%
\[
0<u-u^{\ast}\leq u-u^{\ast}\wedge y=u-u\wedge y^{\ast}=\left(  u-y\right)
^{+}.
\]
It follows that%
\[
0\leq\left(  y-u\right)  ^{+}\wedge\left(  u-u^{\ast}\right)  \leq\left(
y-u\right)  ^{+}\wedge\left(  u-y\right)  ^{+}=0,
\]
so $\left(  y-u\right)  ^{+}\in J$. Let $n\in\mathbb{N}$ satisfy $n\lambda
\geq1$. Hence,%
\[
\left(  \left\vert x\right\vert -nu\right)  ^{+}=n\left(  \frac{1}%
{n}\left\vert x\right\vert -u\right)  ^{+}\leq n\left(  \lambda\left\vert
x\right\vert -u\right)  ^{+}=n\left(  u-y\right)  ^{-}\in J.
\]
Consequently,%
\[
J\left(  \left\vert x\right\vert \right)  \leq nJ\left(  u\right)  \text{ in
}X/J.
\]
This shows that $J\left(  u\right)  $ is a strong unit in $X/J$. In view of
$\mathrm{(U)}$, the vector lattice $X/J$ has a maximal ideal, say
$\mathcal{M}$. But then the preimage $J^{-1}\left(  \mathcal{M}\right)  $ of
$\mathcal{M}$ under $J$ is a maximal ideal of $X$ and the proof is complete.
\end{proof}

\section{Acknowledgments}

This work was supported by the LATOA Grant LR11ES12.

\end{document}